\documentclass[11pt]{article}
\usepackage{setspace,xargs}
\usepackage{amsmath}

\usepackage{graphicx,subfigure}
\usepackage{graphicx,xcolor}
\graphicspath{{./images/}}
 \usepackage{epsfig}

\usepackage [a4paper, bindingoffset=-.5cm, footskip=1cm, headheight = 16pt, top=3cm, bottom=2.5cm, right=3cm, left=3cm ] {geometry}
\input{amssym}
\input{epsf}
\usepackage{stackrel}
\usepackage{algorithm}
\usepackage{algorithmic}
\usepackage{float}
\usepackage{graphicx}
\usepackage{float}
\usepackage{morefloats}
\usepackage{float}
\usepackage{morefloats}
\usepackage{mathtools}

 \usepackage[usenames,dvipsnames]{pstricks}
 \usepackage{epsfig}
 \usepackage{pst-grad} 
 \usepackage{pst-plot} 
 \usepackage[space]{grffile} 
 \usepackage{etoolbox} 
 \makeatletter 
 \patchcmd\Gread@eps{\@inputcheck#1 }{\@inputcheck"#1"\relax}{}{}
 \makeatother
\usepackage[pagewise]{lineno}
\newtheorem{prealg}{{\bf Algorithm}}

\newenvironment{algorithm1}{\begin{prealg}{\hspace{-0.5
em}{\bf.}}}{\end{prealg}}
\newtheorem{prethm}{{\bf Theorem}}

\newenvironment{theorem}{\begin{prethm}{\hspace{-0.5
em}{\bf.}}}{\end{prethm}}
\newtheorem{preex}{{\bf Example}}

\newtheorem{prepr}{{\bf Problem}}

\newenvironment{problem}{\begin{prepr}{\hspace{-0.5
em}{\bf.}}}{\end{prepr}}
\newtheorem{prelem}{{\bf Lemma}}

\newenvironment{lemma}{\begin{prelem}{\hspace{-0.5
em}{\bf.}}}{\end{prelem}}
\newtheorem{precor}{{\bf Corollary}}

\newenvironment{corollary}{\begin{precor}{\hspace{-0.5
em}{\bf.}}}{\end{precor}}
\newtheorem{prepos}{{\bf Proposition}}

\newtheorem{preobv}{{\bf Observation}}

\newtheorem{predef}{{\bf Definition}}

\newtheorem{preproof}{{\bf Proof.}}

\newenvironment{proof}[1]{\begin{preproof}{\rm
               #1}\hfill{$\rule{2mm}{2mm}$}}{\end{preproof}}
               
               \newtheorem{preprooft}{{\bf Proof of Theorem 1.}}

\begin{document}
\renewcommand{\baselinestretch}{1.3}
\title{{\Large\bf A Polynomial Time Algorithm to Find the Star Chromatic Index of Trees}}

{\small
\author{
{\sc Behnaz Omoomi$^a$},
{\sc Elham Roshanbin$^b$}, and
{\sc  Marzieh Vahid Dastjerdi$^a$}
\\ [1mm]
{\small \it $^a$ Department of Mathematical Sciences}\\
{\small \it  Isfahan University of Technology} \\
{\small \it 84156-83111, \ Isfahan, Iran}\\
{\small \it $^b$ Department of Mathematical Sciences}\\
{\small \it  Alzahra University} \\
{\small \it 19938-93973, \ Tehran, Iran}}
\maketitle
\begin{abstract}
\noindent A star edge coloring of a graph $G$ is a proper edge coloring of $G$ such that  every path and cycle of length four in $G$  uses at least three different colors. The star chromatic index of a graph $G$, is the smallest integer $k$ for which $G$ admits a star edge coloring with $k$ colors.  In this paper,  we present a polynomial time algorithm that finds an optimum star edge coloring for every tree. We also provide some tight bounds on the star chromatic index of trees with diameter at most four, and using these bounds we find a formula for the star chromatic index of certain families of trees. \\
\par
\noindent {\bf Keywords:} Star edge coloring, star chromatic index, trees.\\
\noindent {\bf 2010 MSC:} 05C15, 05C05.
\end{abstract}

\section{Introduction}
 A \textit{proper vertex} (\textit{edge coloring}) of a graph $G$ is an assignment of colors to the vertices (edges) of $G$ such that no two adjacent vertices (edges) receive  the same color. 
Under additional constraints on the proper vertex (edge) coloring of graphs, we get a variety of colorings such as the star vertex  and the star edge coloring. A \textit{star vertex coloring} of  $G$, is a proper vertex coloring such that no path or cycle  on four vertices in G is bi-colored (uses at most two colors) \rm{\cite{coleman,fertin}}.  
\par In 2008, Liu and Deng \cite{delta} introduced the edge version of the star vertex coloring that  is defined as follows. A \textit{star edge coloring}   of   $G$ is a proper edge coloring of $G$ such that  no path or cycle of length four (with four edges) in $G$ is bi-colored. We call a star edge coloring of $G$ with $k$ colors, a \textit{$k$-star edge coloring} of $G$. The smallest integer $k$   for which  $G$ admits a $k$-star edge coloring is called the \textit{star chromatic index}  of $G$ and is denoted by  $\chi^\prime_s(G)$.  Liu and Deng \cite{delta} presented  an upper bound on the star chromatic index of graphs with maximum  degree $\Delta\geq 7$. In \cite{mohar},  Dvo{\v{r}}{\'a}k et al.   obtained the  lower bound $2\Delta(1 + o(1))$ and  the near-linear upper bound $\Delta . 2^{O(1)\sqrt{\log\Delta}}$  on the star chromatic index of graphs with maximum degree $\Delta$.  They also presented some upper bounds  and lower bounds on the star chromatic index of  complete graphs  and  subcubic graphs (graphs with maximum degree at most 3). In \rm{\cite{class}}, Bezegov{\'a}  et al.   obtained some bounds on the star chromatic index of subcubic outerplanar graphs, trees and outerplanar graphs (see also \cite{Kerdjoudj, Lie,luzar, pradeep,wang}). 

\par In this paper, by a polynomial time algorithm, we  determine the  star chromatic index of  every tree. For this purpose, we first define a {\it Havel-Hakimi} type problem. The Havel-Hakimi problem, is a problem in which we are asked to determine whether or not there exists a simple graph with a given degree sequence (a sequence of the vertex degrees) \cite{havel}. In \cite{erdosh}, Erd{\"o}s et al. extended the Havel-Hakimi  problem to the problem of  existence of simple digraphs (there are no two edges with the same direction  between any two vertices, but loops are allowed)  possessing some prescribed bi-degree sequences (a sequence of the vertex outdegrees and indegrees).
 In the Havel-Hakimi type problem that we define in this paper, we determine whether it is possible to construct an {\it oriented graph} (a digraph that its underlying graph is simple) with a given outdegree sequence (a sequence of vertex outdegrees)  without caring about the indegrees. 
With a similar idea in \cite{erdosh, havel}, we present an algorithm that finds a solution for this problem in polynomial time. Then, we show  that
this Havel-Hakimi type problem is indeed polynomially equivalent to the problem of existence of a star edge coloring of a tree with  diameter at
most four (or a $2H$-tree for short), with specific number of colors.  Using this equivalency, we present a polynomial time algorithm that determines the star chromatic index of $2H$-trees by finding an optimum star edge coloring of them. We then give a polynomial time algorithm that extends the optimum star edge coloring
of $2H$-trees to an optimum star edge coloring of trees in general. 

\par This paper is organized as follows. 
In Section~\ref{pre}, we briefly introduce some graph theory   terminology and notations that we use in this paper.  In Section~\ref{digraph},  we define a Havel-Hakimi type problem, and we   give a greedy algorithm that finds a solution to this problem. 
   In Section \ref{2tree}, we give a polynomial time algorithm to determine the star chromatic index of every $2H$-tree. To do this, we first prove that the problem of  existence a star edge coloring for $2H$-trees is  polynomially equivalent with the  Havel-Hakimi type problem defined in Section~\ref{digraph}.   In Section~\ref{tree}, we show that  finding the star chromatic index of $2H$-trees leads to determining the star chromatic index of every tree. Moreover, we define a polynomial time algorithm  that  provides an optimum star edge coloring  for every tree.
 In Section~\ref{upper}, we present some tight bounds on the star chromatic index of $2H$-trees. Using these bounds we  find a formula for the star chromatic index of certain $2H$-trees and the caterpillars  (a \emph{caterpillar} is a tree for which  removing  the leaves produces a path).
\section{Preliminaries}\label{pre}
In this section, we present the terminology and notations that we use in this paper. 
 For a vertex $v$ of a graph $G$, we denote the degree of $v$ by $d_G(v)$.
In a digraph $G$,  the number of edges going into a vertex $v$, denoted by $d^-_G(v)$, is known as the {\it indegree} of $v$ and the number of edges  coming out of $v$, denoted by $d^+_G(v)$, is known as the {\it outdegree} of $v$. The set of   in-neighbours and out-neighbours  of $v$ are denoted by  $N_G^-(v)$ and $N^+_G(v)$, respectively.
When $G$ is clear from the context, we simply write $d(v)$,  $d^-(v)$, $d^+(v)$,  $N^-(v)$, and $N^+(v)$. For every vertex $u$ and $v$ of digraph $G$, by $\overrightarrow{uv}$, we mean a directed edge from $u$ to $v$. 
 For further information on graph theory concepts and terminology we refer the reader to \cite{bondy}.
\par 
A finite sequence  $d^+ = ((d_1^+,v_1),\ldots,(d_n^+,v_n))$ of ordered pairs $(d_i^+,v_i)$, $1\leq i \leq n$, in which $d_i^+$ is a non-negative integer and $v_i$ represents a vertex, is called an {\it   outdegree-vertex sequence} (or OVS for short). An OVS $d^+$ is called an {\it   outdegree-vertex graphical sequence} (or OVGS  for short) if there exists an   oriented graph $G$ with vertex set $\{v_1,  \ldots , v_n\}$, such that the outdegree of vertex $v_i$ is $d^+_i$, $1\leq i\leq n$. In this case, we say that $G$ {\it realizes} $d^+$, or $G$ is a {\it realization} of $d^+$. Note that the only condition on the  indegree sequence of $G$ is that $\sum_{i}{d^+_i}=\sum_{i}{d^-_i}$. For example, a realization of OVS $d^+=((2,v_1),(3,v_2),(3,v_3),(0,v_4),(0,v_5))$ is shown in Figure~\ref{fig1}. Hence, $d^+$ is an OVGS.
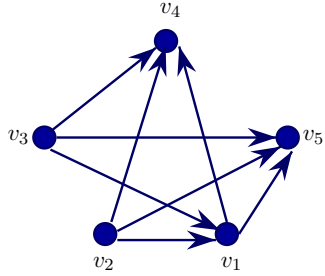
\begin{figure}[H]
\begin{center}
\psscalebox{0.8 0.8} 
{
\begin{pspicture}(0,-2.2)(5.28,2.2)
\definecolor{colour0}{rgb}{0.0,0.0,0.4}
\definecolor{colour1}{rgb}{0.0,0.0,0.6}
\pscircle[linecolor=colour0, linewidth=0.04, fillstyle=solid,fillcolor=colour1, dimen=outer](2.6,1.6){0.2}
\pscircle[linecolor=colour0, linewidth=0.04, fillstyle=solid,fillcolor=colour1, dimen=outer](4.6,0.0){0.2}
\pscircle[linecolor=colour0, linewidth=0.04, fillstyle=solid,fillcolor=colour1, dimen=outer](0.6,0.0){0.2}
\pscircle[linecolor=colour0, linewidth=0.04, fillstyle=solid,fillcolor=colour1, dimen=outer](1.6,-1.6){0.2}
\pscircle[linecolor=colour0, linewidth=0.04, fillstyle=solid,fillcolor=colour1, dimen=outer](3.6,-1.6){0.2}
\psline[linecolor=colour0, linewidth=0.04, arrowsize=0.12cm 5.0,arrowlength=0.5,arrowinset=2.88]{->}(0.7,0.1)(2.2,1.3)
\psline[linecolor=colour0, linewidth=0.04, arrowsize=0.12cm 5.0,arrowlength=0.5,arrowinset=3.0]{->}(0.8,0.0)(4.1,0.0)
\psline[linecolor=colour0, linewidth=0.04, arrowsize=0.12cm 5.0,arrowlength=0.5,arrowinset=3.0]{->}(0.7,-0.2)(3.2,-1.4)
\psline[linecolor=colour0, linewidth=0.04, arrowsize=0.12cm 5.0,arrowlength=0.5,arrowinset=3.0]{->}(1.7,-1.4)(2.5,1.2)
\psline[linecolor=colour0, linewidth=0.04, arrowsize=0.12cm 5.0,arrowlength=0.5,arrowinset=3.0]{->}(1.8,-1.55)(4.2,-0.3)
\psline[linecolor=colour0, linewidth=0.04, arrowsize=0.12cm 5.0,arrowlength=0.5,arrowinset=3.0]{->}(1.8,-1.7)(3.1,-1.7)
\psline[linecolor=colour0, linewidth=0.04, arrowsize=0.12cm 5.0,arrowlength=0.5,arrowinset=3.0]{->}(3.6,-1.4)(2.9,1.2)
\psline[linecolor=colour0, linewidth=0.04, arrowsize=0.12cm 5.0,arrowlength=0.5,arrowinset=3.0]{->}(3.8,-1.6)(4.5,-0.5)
\rput[bl](2.5,2.0){$v_4$}
\rput[bl](4.85,-0.1){$v_5$}
\rput[bl](3.5,-2.2){$v_1$}
\rput[bl](1.4,-2.2){$v_2$}
\rput[bl](0.0,-0.1){$v_3$}
\end{pspicture}
}
\caption{A realization of OVGS $d^+=((2,v_1),(3,v_2),(3,v_3),(0,v_4),(0,v_5))$.}
\label{fig1}
\end{center}
\end{figure}
\par 
 Let $X$ be a  finite sequence of objects.
 We denote  the length of $X$ by $|X|$ and the $i$-th element of $X$ by $X[i]$, $1\leq i\leq |X|$.  
 For simplicity, we denote the subsequence $X[i],X[i+1],\ldots, X[i+a]$ of $X$  by $X[i..i+a]$, where   $1\leq i\leq  |X|$ and $0\leq a\leq |X|-i$. 
 If each element of $X$ is an ordered pair, then we denote the $j$-th coordinate of the $i$-th element of $X$, by $X[i][j]$, where  $1\leq i\leq |X|$ and $j\in \{1,2\}$.
  For a subset  $A=\{X[i_1][2],\ldots,X[i_k][2]\}$  of  $X^2=\{X[i][2]:1\leq i\leq |X|\}$, we denote the vector $(i_1,\ldots,i_k)$ by $s(A)$, where $i_1\leq\cdots \leq i_k$.
   For example, if $X=((a,w),(b,x),(c,y),(d,z))$ and $A=\{w,y,z\}$, then  $s(A)=(1,3,4)$.
   
   We define a partial order ``$\preceq$''  among $k$-element finite sequences of positive integers as follows. We say $a\preceq b$ if  for each $1\leq i\leq k$, $a[i] \leq  b[i]$.
    Let $A$ and $B$ be two subsets of   $X^2$. We write
\[B \leq A~\text{if and only if}~s(B)\preceq s(A)\] 
and we say that $B$ is to the left of $A$.
For example, let $X=((a,u),(b,v),(c,w),(d,x),(e,y),(f,z))$, $A=\{v,x,z\}$, and $B=\{v,w,y\}$. Then, $s(A)=(2,4,6)$, $s(B)=(2,3,5)$, and $s(B)\preceq s(A)$. Thus, $B$ is to the left of $A$.
\par 
A {\it rooted tree} is a tree in which one vertex has been designated as the root. The {\it height} of a rooted tree $T$ is the number of edges on the longest  path between the root and a leaf.  The {\it level} of a vertex in $T$ is the distance between the vertex and the root plus one. Note that the level of the root is one. If $v$ is a vertex in level $\ell>1$ of  $T$ and $u$  is the parent of $v$ (its neighbour in level $\ell-1$), then we denote the set of  edges incident to $v$, except  $uv$, by $E_{v}$. 

If $T$  is a tree with diameter at most four, then we say that $T$ is a {\it $2H$-tree}. In other words, in $T$ there exists a vertex $u$  such that the height of $T$ with root $u$ is at most two.
Suppose that $d(u)=t$ and $u_1, \ldots , u_t$ are the neighbours of $u$.
For each $1\leq i\leq t$, let  $n_i$  be the size of  $E_{u_i}$ (i.e., $d(u_i)=n_i+1$),
 then we denote the $2H$-tree $T$ by $T_{n_1, \ldots, n_t}$, where  $n_1\leq n_2\leq \cdots\leq n_t$.
We call a $2H$-tree   in which all neighbours of the root
are of the same degree $r>0$, an  \emph{$r$-regular $2H$-tree} and denote by $T_{(r,t)}$.
\section{Realization of outdegree-vertex sequences}\label{digraph}
 In this section, we present a construction algorithm that determines whether a given outdegree-vertex sequence (or OVS) is an outdegree-vertex graphical sequence (or an OVGS) or not, and for an OVGS $d^+$ provides an oriented graph that  realizes $d^+$.

 We need the following notations and definitions to state the results of this section.
Suppose that $d^+=((d_1^+,v_1),\ldots, (d_n^+,v_n))$ is an OVS and $W$ is a proper subset of $V=\{v_1,\ldots,v_n\}$. Let $G_W$ be an  oriented graph on $V$ with the following outdegrees.
\[d^+_{G_W}(v_i)=\begin{cases}
d_i^+&~\text{if}~v_i\in W,\\
0&~\text{otherwise.}
\end{cases}
\]
 We say that $d^+$ is {\it normal} if $d^+_i\leq d^+_{i+1}$,  for every $1\leq i\leq n$. We also say that $d^+$ is {\it $G_W$-normal} if it has the following properties.
\begin{itemize}
\item For every $1\leq i\leq |W|$, $v_i\in W$.
\item For every $|W|< i\leq n$,  $d_i^++d^-_{G_W}(v_i)< d_{i+1}^++d^-_{G_W}(v_{i+1})$, or\\ $d_i^++d^-_{G_W}(v_i)= d_{i+1}^++d^-_{G_W}(v_{i+1})$ and $d^+_{i}\leq d^+_{i+1}$.
\end{itemize}

A {\it possible out-neighbour} (or PON for short) of $v_i\in V\setminus W$ is a subset of $V\setminus (N^-_{G_W}(v_i)\cup\{v_i\})$ with $d_i^+$ elements that is a candidate for being  the out-neighbours of $v_i$ in a realization of $d^+$.   
The {\it leftmost PON}  of $v_i$, denoted by $L_{{G_W}}(v_i)$, is  a PON of $v_i$ such that for every PON $A$ of $v_i$, $L_{G_W}(v_i)\leq A$. Indeed, $L_{G_W}(v)$ is the subset of $V\setminus (N^-_{G_W}(v_i)\cup\{v_i\})$  that contains $d^+_i$ vertices with the smallest subscripts.

\begin{theorem}\label{th1}
Let $d^+=((d_1^+,v_1),\ldots,(d_n^+,v_n))$ be an OVS  and $W$ be a proper subset of the vertex set $V=\{v_1,\ldots,v_n\}$. Suppose that for an oriented graph $G_W$, $d^+$ is $G_W$-normal and  for some $1\leq i\leq n$, there exists vertex $v_i\in V\setminus W$ with $d^+_i>0$. 
The OVS $d^+$ has realization $G$ in which for every $v\in W$, $N^+_G(v)=N^+_{G_W}(v)$ if and only if  $d^+$ has a realization $H$ in which for every $v\in W$, $N^+_H(v)=N^+_{G_W}(v)$ and $N^+_{H}(v_i)=L_{G_W}(v_i)$.
\end{theorem}

\begin{proof}
{
Let $G$ be a realization of  $d^+$ and for every $v\in W$, $N^+_G(v)=N^+_{G_W}(v)$. If $N^+_G(v_i)=\nolinebreak L_{G_W}(v_i)$, then we are done. Otherwise, we will present a sequence of changes in the edges of $G$ that preserve the out-neighbours of every vertex in $W$ and convert $G$ into a graph $H$ in which  $N^+_{H}(v_i)=L_{G_W}(v_i)$.
There is an increasing bijective function $\phi :s(N^+_G(v_i)\setminus L_{G_W}(v_i))\rightarrow s(L_{G_W}(v_i)\setminus N^+_G(v_i))$ because of $|N^+_G(v_i)|=|L_{G_W}(v_i)|$.
Thus, the function  $\Psi :N^+_G(v_i)\setminus L_{G_W}(v_i)\rightarrow L_{G_W}(v_i)\setminus N^+_G(v_i)$, with $\Psi(v_j)=v_{\phi(j)}$, for every $v_j\in N^+_G(v_i)\setminus L_{G_W}(v_i)$,  is bijective such that  $\phi(j)\leq j$.
The last inequality holds, since $L_{G_W}(v_i)$ is the leftmost PON of~$v_i$.

Now suppose that $v_j\in N^+_G(v_i)\setminus L_{G_W}(v_i)$, $v_k\in L_{G_W}(v_i)\setminus N^+_G(v_i)$, and   $\Psi(v_j)=v_k$, where $\Psi$ is the function that we defined above.
 Let $A = (N^+_G(v_i) \setminus \{v_j\})\cup \{v_k\}$. We construct another realization $G^\prime$ of $d^+$ such that $N^+_{G^\prime}(v_i)=A$. 

Since $v_j\in N^+_G(v_i)$ and $v_k \not\in N^+_G(v_i)$, we have $\overrightarrow{v_iv_j}\in E(G)$ and $\overrightarrow{v_iv_k}\not\in E(G)$.  We have two possibilities: either  $\overrightarrow{v_kv_i}$ is an edge of $G$ or not. If $\overrightarrow{v_kv_i}\not\in E(G)$, then by adding edge $\overrightarrow{v_iv_k}$ and removing edge $\overrightarrow{v_iv_j}$,  we achieve the desired realization. Now, we suppose that $\overrightarrow{v_kv_i}\in E(G)$. Thus, we conclude that $v_k\in V\setminus W$, since  $v_k\in L_{G_W}(v_i)$. Also $v_j\in V\setminus W$, since $d^+$ is $G_W$-normal,  and $k<j$. We again have two possibilities: either $v_k$ and $v_j$ are connected by an edge or not.  If there is  no edge between $v_k$ and $v_j$, then we create the required graph by removing the edges $\overrightarrow{v_kv_i}$ and $\overrightarrow{v_iv_j}$ and adding the edges $\overrightarrow{v_iv_k}$ and $\overrightarrow{v_kv_j}$. Now assume that  one of the edges  $\overrightarrow{v_jv_k}$ or $\overrightarrow{v_kv_j}$ belongs to $E(G)$. If $\overrightarrow{v_jv_k}\in E(G)$, then we reverse the directions of the edges  $\overrightarrow{v_kv_i}$, $\overrightarrow{v_iv_j}$, and $\overrightarrow{v_jv_k}$. Thus, suppose that  $\overrightarrow{v_kv_j}\in E(G)$. If there exists  vertex $v_m$ such that there is no edge between $v_m$  and $v_k$, then we add edge $\overrightarrow{v_kv_m}$, reverse the direction of the edge $\overrightarrow{v_kv_i}$ and remove edge $\overrightarrow{v_iv_j}$. Otherwise, there exist an edge between $v_k$ and every vertex in $V\setminus\{v_k\}$. 
 Note that in this case, two out-neighbours of $v_k$ and two in-neighbours of $v_j$ in $V\setminus W$ are determined. Thus, because of $d^+_k+|N^-_{G_W}(v_k)|\leq d^+_j+|N^-_{G_W}(v_j)|$, there exists vertex $v_m\in V\setminus W$ such that edges $\overrightarrow{v_jv_m}$ and $\overrightarrow{v_mv_k}$ belong to $E(G)$. Therefore, it suffices  to  reverse the directions of the edges $\overrightarrow{v_kv_i}$, $\overrightarrow{v_iv_j}$, $\overrightarrow{v_jv_m}$, and $\overrightarrow{v_mv_k}$. Thus, in  all cases we obtain the required realization.
 
 We now apply this process for each $v_j \in N^+_G(v_i) \setminus L_{G_W}(v_i)$  to exchange $v_j$ with $\Psi(v_j)$ such that at every step in the obtained graph $\Psi(v_j)\in N^+(v_i)$.
 After the last step,  the final graph is $H$ and $L_{G_W}(v_i)=N^+_H(v_i)$, as desired. The converse of the statement is trivial.}
\end{proof}
Using Theorem \ref{th1}, we now present a construction algorithm that determines whether a given OVS is an OVGS or not, and in the  case that it is, it provides a  realization of it.

\begin{theorem}\label{th2}
For every  OVS $d^+$, there is a polynomial time algorithm that determines whether $d^+$ is an OVGS or not, and if so finds  a realization of it.
\end{theorem}

\begin{proof}
{
Let $d^+=((d_1^+,v_1),\ldots,(d_n^+,v_n))$ be an OVS. The following algorithm determines whether $d^+$ is an OVGS or not. Moreover, if $d^+$ is an OVGS, then the algorithm finds a realization of $d^+$ such that in each step, for every $1\leq i\leq n$ the set of out-neighbours of $v_i$ is its leftmost PON.
 \vspace*{5mm}
\hrule
\vspace*{-5mm}
\begin{algorithm1}\label{digraph1}
Recognition and realization of the given OVS $d^+=((d_1^+,v_1),\ldots,(d_n^+,v_n))$.
\vspace*{.7mm}
\hrule
\medskip
{\bf Step 1.} Normalize the given OVS $d^+$.

\medskip

{\bf Step 2.} Set  $i_0=1$ and $W=\emptyset$.

\medskip

{\bf Step 3.} While $i_0\leq n$ and $d^+[i_0][1]=0$, set $i_0=i_0+1$ and $W=W\cup\{d^+[i_0][2]\}$.

\medskip

{\bf Step 4.} Let  $G_W$ be a graph with no  edges on vertex set $\{v_1,\ldots,v_n\}$.

\medskip
{\bf Step 5.} For $i$ from $i_0$ to $n$  do the following steps.

\medskip

{\setlength\parindent{30pt}{\bf Step 5.1.} $G_W$-normalize $d^+$.

\medskip
%
%
%
%

{\bf Step 5.2.} Set $A_{i}=\{v_1,\ldots,v_n\}\setminus (N^-_{G_W}(d^+[i][2])\cup\{d^+[i][2]\})$.

\medskip

{\bf Step 5.3.} If $|A_{i}|< d^+[i][1]$, then print ``No" and stop.

\medskip

{\bf Step 5.4.} If $|A_{i}|\geq d^+[i][1]$, then call the set of $d^+[i][1]$ vertices in $A_i$ with the smallest\\
\hspace*{75pt} subscripts in $d^+$ as  $L_{G_W}(d^+[i][2])$.
\medskip

{\bf Step 5.5.} For every $v$ in $L_{G_W}(d^+[i][2])$, conncet $d^+[i][2]$ to $v$.

\medskip

{\bf Step 5.6.} Set $W=W\cup\{d^+[i][2]\}$.}

\medskip
 
{\bf Step 6.} Return the obtained oriented graph. 
\vspace*{1mm}
\hrule
\end{algorithm1}

In Step~1 of the algorithm, we first arrange the elements of $d^+$ such that for every $1\leq i\leq n$, $d^+[i][1]\leq d^+[i+1][1]$ (normalizing $d^+$).
In Step~2, we define variable $i_0$ with initial value one and empty set $W$. In Step~3, we increase $i_0$ to the smallest subscript for which vertex $d^+[i_0][2]$ has positive outdegree $d^+[i_0][1]$. Moreover,  we add every vertex with zero outdegree to $W$.
In Step~4, we consider  graph $G_W$ on vertex set $\{v_1,\ldots,v_n\}$  without any edges.  During the algorithm we extend $G_W$ to a realization of $d^+$ (if possible) by  adding some edges such that  the out-neighbours of vertices in $W$ are preserved. 
Namely, in Step~5,  for $i$ from $i_0$ to $n$, we  determine the out-neighbours of vertex $d^+[i][2]$, while the out-neighbours of all vertices with subscripts less than $i$  are already identified.
 In Step~5.1,  we rearrange $d^+$ such  that it is $G_W$-normal, if necessary. 
   
   In Step~5.2, we obtain the set of  allowed out-neighbours for $d^+[i][2]$,  and we denote this set by~$A_{i}$. Note that, by Theorem~\ref{th1} if $d^+$ is an OVGS, then there exists a realization $G$ of it such that, for every $1\leq j\leq i-1$,  $N^+_{G}(d^+[j][2])=N_{G_W}(d^+[j][2])$ and $N^+_{G}(d^+[i][2])=L_{G_W}(d^+[i][2])$. Hence, according to the size of $A_i$, we implement Step~5.3 and 5.4 as follows.
   In Step~5.3, if size of $A_{i}$ is less than $d^+[i][1]$, then we conclude that there is no realization for the given OVS $d^+$. Thus, the algorithm prints "No" and  stops the algorithm.
   Otherwise, in Step~5.4 we determine the elements of $L_{G_W}(d^+[i][2])$ as the  leftmost PON of $d^+[i][2]$. In Step~5.5, we connect  $d^+[i][2]$ to the vertices in $L_{G_W}(d^+[i][2])$. In Step~5.6, we add vertex $d^+[i][2]$ to $W$.
 Finally,    in Step~6, if $d^+$ is an OVGS, then the algorithm returns the realization  of $d^+$.
 
 We now prove that Algorithm~\ref{digraph1} is a polynomial time algorithm. In Step~1, 5.1 and 5.4, we have to use merge sorting and therefore these steps are of order $O(n\log n)$.  Step~3, 5.2 and 5.5 are single scans and therefore, the running time of  these steps is $O(n)$.   
 The running time of the other steps of the algorithm, is $O(1)$.  
 Since Step~5, runs at most $n$ times, then the time complexity of the algorithm is $O(n^2\log n)$.
 }\end{proof}
\section{Star chromatic index of $2H$-trees}\label{2tree}
In this section, using the results in Section~\ref{digraph1}, we give a polynomial time algorithm to find the star chromatic index of $2H$-trees. For this purpose, we first show the equivalency between the following two problems.
  
\begin{problem}\label{p1}\\
\textbf{Given}: a $2H$-tree, $T_{n_1, \ldots, n_t}$.\\
\textbf{Find}: minimum integer $k$ for which there is a star edge coloring of  $T_{n_1, \ldots, n_t}$ with $t+k$  colors.
\end{problem}

\begin{problem}\label{p2}\\
\textbf{Given:} an OVS $d^+=((n_1, v_1),\ldots, (n_t, v_t))$.\\
\textbf{Find:}  minimum integer $k$ for which OVS  $D^+_k=((0,v_{t+1}),\ldots,(0,v_{t+k}),(n_1,v_1),\ldots,(n_t,v_t))$ is  an OVGS.
\end{problem}

Note that, since the root of  $2H$-tree $T_{n_1,\ldots,n_t}$ is a vertex of degree $t$, we have $\chi^\prime_s(T_{n_1,\ldots,n_t})\geq t$.
    Hence, without loss of generality we can assume that $\chi^\prime_s(T_{n_1,\ldots,n_t})=t+k$, for some non-negative integer $k$.
     Thus, finding the star chromatic index of $T_{n_1,\ldots,n_t}$ is in fact equivalent to finding the minimum $k$ for which there is a $(t+k)$-star edge coloring of $T_{n_1,\ldots,n_t}$.
      In the following theorem, we prove that  finding the minimum desired $k$, is polynomially equivalent to finding the minimum $k$ for which  OVS $D^+_k=((0,v_{t+1}),\ldots,(0,v_{t+k}),(n_1,v_1),\ldots,(n_t,v_t))$ is  an OVGS.  
      
\begin{theorem}\label{p1p2}
Problem~\ref{p1} and Problem~\ref{p2} are polynomially equivalent.
\end{theorem}
\begin{proof}
{
  First assume that $c$  is a star edge coloring of $T_{n_1, \ldots,n_t}$, with color set $C=\{1, \ldots,t+k\}$. Let vertex $u$ be the root of $T_{n_1,\ldots,n_t}$ and  $u_1,\ldots,u_t$ be the neighbours of $u$.  Up to renaming colors, we can assume that $c(uu_i)=i$, for every $1\leq i\leq t$.  
We now  construct a digraph $G$ with the following vertex~set and edge set.
  \[V(G)=\{v_i: 1\leq i\leq t+k\},\quad(\text{corresponding to colors in}~C)\]
  \[ E(G)=\{ \overrightarrow{v_{i}v_j}:~i\neq j,~1\leq i\leq t,~1\leq j\leq k+t,~\text{and there is an edge with color}~j~\text{in}~E_{u_i}\}.\]
  
   Since $c$ is a proper edge coloring, for every       $1\leq i\leq t$, $c$ uses  $n_{i}$ (the size of $E_{u_i}$) different colors from $C\setminus\{i\}$ for coloring the edges in $E_{u_i}$.  Therefore, for every $1\leq i\leq t$,  $d^+(v_{i})=n_i$ and the rest of the vertices in $G$ have zero outdegree. Moreover,  there are no loops and no two edges with the same direction between any two vertices in $G$. Also, since $c$ is a star edge coloring, for every $i$ and $j$ in $\{1, \ldots,t\}$, if color $j$ appears in $E_{u_i}$, then color $i$ cannot appear in $E_{u_j}$, while each color in  $\{t+1,\ldots,t+k\}$ can be used in every $E_{u_i}$. Therefore, for every $i$ and $j$, where $i\neq j$ and $1\leq i,j\leq t+k$, $G$ contains at most one of the edges $\overrightarrow{v_iv_j}$  and $\overrightarrow{v_jv_i}$. Hence, $G$ is a realization of~$D^+_k$.
  \par
  Conversely, assume that  $G$ is a realization of  $D^+_k$. Then, for each    $1\leq i\leq t$, vertex $v_{i}$ has $n_i$ different out-neighbours. 
  Moreover, if for some $j\in C$, $\overrightarrow{v_{i}v_j}\in E(G)$, then $\overrightarrow{v_jv_{i}}\not\in E(G)$. 
 Now, we present an edge coloring $c$ for $T_{n_1,\ldots,n_t}$ as follows.
    For every     $1\leq i\leq t$,  we define  $c(uu_i)=i$ and color the edges of $E_{u_i}$ with different elements of  $\{j:  \overrightarrow{v_{i}v_j}\in E(G)\}$.
     This edge coloring is a star edge coloring of $T_{n_1,\ldots,n_t}$, because if there exists  a bi-colored path, say $xu_iuu_jy$, then $c(xu_i)=c(uu_j)=j$ and $c(yu_j)=c(uu_i)=i$. 
     Therefore, by definition of $c$,  both edges $\overrightarrow{v_{i}v_j}$ and $\overrightarrow{v_{j}v_i}$  must belong to  $E(G)$, which is a contradiction. Thus,  $c$ is a star edge coloring of $T_{n_1,\ldots,n_t}$.
It is easy to see that the above argument provides a polynomial time reduction from Problem~\ref{p1} to Problem~\ref{p2} and vice versa. 
}
\end{proof}

By proof of Theorem~\ref{p1p2},  given an OVGS $D^+_k=((0,v_{t+1}),\ldots,(0,v_{t+k}),(n_1,v_1),\ldots,(n_t,v_t))$ with realization $G$, we can find a star edge coloring of $T_{n_1,\ldots,n_t}$ (with root $u$) in which for every $1\leq i\leq t$, the color of $uu_i$ is $i$ and the color set of the edges in $E_{u_i}$  corresponds to $N^+_G(v_i)$. For example,  assume that $k=2$ and $D^+_2=((0,v_{4}),(0,v_{5}),(2,v_1),(3,v_2),(3,v_3))$ is an OVGS with realization $G$,  shown in Figure~\ref{fig1}. Then, $n_1=2$  and $n_2=n_3=3$. By proof of Theorem~\ref{p1p2}, if for every $1\leq i\leq 3$, we color edge $uu_i$ in $T_{2,3,3}$ with $i$, and color an edge in $E_{u_i}$ with color $j$, wherever $\overrightarrow{v_iv_j}$ is an edge in $G$, then the obtained coloring is a star edge coloring of $T_{2,3,3}$, as demonstrated in Figure~\ref{fig2}.  
\begin{figure}[!ht]
\begin{center}
\psscalebox{0.8 0.8} 
{
\begin{pspicture}(0,-1.7083334)(9.133333,1.7083334)
\definecolor{colour0}{rgb}{0.0,0.0,0.4}
\definecolor{colour1}{rgb}{0.0,0.0,0.6}
\pscircle[linecolor=colour0, linewidth=0.04, fillstyle=solid,fillcolor=colour1, dimen=outer](4.05,1.2916666){0.16666667}
\pscircle[linecolor=colour0, linewidth=0.04, fillstyle=solid,fillcolor=colour1, dimen=outer](7.133333,0.20833333){0.16666667}
\pscircle[linecolor=colour0, linewidth=0.04, fillstyle=solid,fillcolor=colour1, dimen=outer](4.05,-0.20833333){0.16666667}
\pscircle[linecolor=colour0, linewidth=0.04, fillstyle=solid,fillcolor=colour1, dimen=outer](1.3833333,0.20833333){0.16666667}
\pscircle[linecolor=colour0, linewidth=0.04, fillstyle=solid,fillcolor=colour1, dimen=outer](0.16666667,-0.675){0.16666667}
\pscircle[linecolor=colour0, linewidth=0.04, fillstyle=solid,fillcolor=colour1, dimen=outer](1.2666667,-1.2916666){0.16666667}
\pscircle[linecolor=colour0, linewidth=0.04, fillstyle=solid,fillcolor=colour1, dimen=outer](2.8,-1.275){0.16666667}
\pscircle[linecolor=colour0, linewidth=0.04, fillstyle=solid,fillcolor=colour1, dimen=outer](4.0666666,-1.5416666){0.16666667}
\pscircle[linecolor=colour0, linewidth=0.04, fillstyle=solid,fillcolor=colour1, dimen=outer](5.266667,-1.275){0.16666667}
\pscircle[linecolor=colour0, linewidth=0.04, fillstyle=solid,fillcolor=colour1, dimen=outer](8.266666,-0.975){0.16666667}
\pscircle[linecolor=colour0, linewidth=0.04, fillstyle=solid,fillcolor=colour1, dimen=outer](6.866667,-1.275){0.16666667}
\pscircle[linecolor=colour0, linewidth=0.04, fillstyle=solid,fillcolor=colour1, dimen=outer](8.966666,0.125){0.16666667}
\psline[linecolor=colour0, linewidth=0.04](3.8833334,1.2083334)(1.4666667,0.29166666)
\psline[linecolor=colour0, linewidth=0.04](1.3,0.125)(0.21666667,-0.625)
\psline[linecolor=colour0, linewidth=0.04](4.05,1.125)(4.05,-0.041666668)
\psline[linecolor=colour0, linewidth=0.04](3.9833333,-0.29166666)(2.9,-1.2083334)
\psline[linecolor=colour0, linewidth=0.04](4.05,-0.375)(4.05,-1.375)
\psline[linecolor=colour0, linewidth=0.04](4.133333,-0.29166666)(5.133333,-1.125)
\psline[linecolor=colour0, linewidth=0.04](7.133333,0.041666668)(6.883333,-1.125)
\psline[linecolor=colour0, linewidth=0.04](7.2166667,0.125)(8.133333,-0.875)
\psline[linecolor=colour0, linewidth=0.04](7.2166667,0.20833333)(8.8,0.125)
\psline[linecolor=colour0, linewidth=0.04](1.3833333,0.125)(1.3,-1.125)(1.3,-1.125)
\rput[bl](2.55,0.85833335){$1$}
\rput[bl](4.2,0.34166667){$2$}
\rput[bl](5.5,0.84166664){$3$}
\rput[bl](4.85,-0.7416667){$5$}
\rput[bl](8.05,0.25833333){$5$}
\rput[bl](0.35,-0.24166666){$4$}
\rput[bl](1.45,-0.7416667){$5$}
\rput[bl](3.05,-0.7416667){$1$}
\rput[bl](6.55,-0.7416667){$1$}
\rput[bl](7.35,-0.7416667){$4$}
\rput[bl](3.65,-1.1416667){$4$}
\psline[linecolor=colour0, linewidth=0.04](4.2166667,1.2083334)(6.9666667,0.29166666)
\rput[bl](3.95,1.5583333){$u$}
\rput[bl](1.15,0.45833334){$u_1$}
\rput[bl](3.45,-0.14166667){$u_2$}
\rput[bl](7,0.45833334){$u_3$}
\end{pspicture}
}
\caption{ A star edge coloring of $T_{2,3,3}$ corresponding to the graph shown in Figure~\ref{fig1}.}
\label{fig2}
\end{center}
\end{figure}
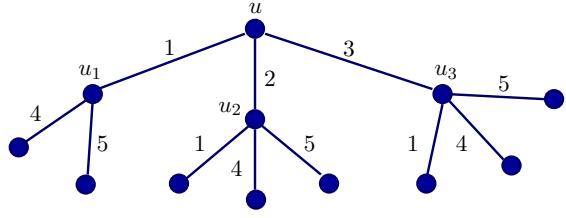

Now using Theorem~\ref{th1}, \ref{th2}, and \ref{p1p2}, we are ready to propose a polynomial time algorithm to solve Problem~\ref{p1}.
\begin{theorem}
There is a polynomial time algorithm for computing the star chromatic index of every $2H$-tree  and presenting an optimum star edge coloring of it.
\end{theorem}
\begin{proof}
{
In the following algorithm, we  present an optimum star edge coloring for the given $2H$-tree $T_{n_1,\ldots,n_t}$.
\vspace*{5mm}
\hrule
\vspace*{-4.5mm}
\begin{algorithm1}\label{2H}
An optimum star edge coloring of the given $T_{n_1,\ldots,n_t}$ with root $u$.
\vspace*{1.5pt}
\hrule
\medskip
{\bf Step 1.} Set $k=0$ and $D^+_k=((n_1,v_1),\ldots,(n_t,v_t))$.

\medskip

{\bf Step 2.} While Algorithm~\ref{digraph1} returns ``No" for the given OVS $D^+_k$, \\
{\hspace*{53pt} set $k=k+1$ and $D^+_k=((0,v_{t+k}),D^+_{k-1}[1..k+t-1])$.}

\medskip

%
{\bf Step 3.} Let $G$ be the realization of $D^+_k$ obtained in the last implement of  previous step. \\ 
\hspace*{53pt}For $i$ from 1 to $t$ do the following steps. 

\medskip

{\setlength\parindent{30pt} {\bf Step 3.1.} Color edge $uu_i$ of $T_{n_1,\ldots,n_t}$ with  $i$}. 

\medskip

{\setlength\parindent{30pt} {\bf Step 3.2.}  Color the edges of $E_{u_i}$ with different  subscripts of the vertices in $N^+_G(v_i)$.
}
\medskip

{\bf Step 4.} Return the value of $k$ and the obtained  edge coloring as a star edge coloring of $T_{n_1,\ldots,n_t}$\\
\hspace*{53pt}  with $t+k$ colors.
\vspace*{1mm}
\hrule
\end{algorithm1}

In Step 1 and 2 of Algorithm~\ref{2H},  we find the minimum $k$ for which $D^+_k$ is an OVGS.
  In Step~1 of Algorithm~\ref{2H},  we define OVS $D^+_k$ and  variable $k$ (with initial value zero) that  its final value in the algorithm  is the answer to Problem~\ref{p1}. The value of $k$, only changes if  in Step~2  Algorithm~\ref{digraph1} dose not return a realization of $D^+_k$. In such a case,  we increase $k$ by one, and we add the ordered pair $(0,v_{t+k})$ to OVS $D^+_{k-1}$.
 In Step~3, we take the realization $G$ of $D_k$ (obtained in the previous step for the final value of $k$), and   by Theorem~\ref{p1p2} we use the subscripts of the out-neighbours  of each vertex  in $G$ to define a star edge coloring for $T_{n_1,\ldots,n_t}$. 
In Step~4, Algorithm~\ref{2H} returns the final value of $k$ and an optimum star edge coloring of $T_{n_1,\ldots,n_t}$

Note that  only Step~2 and 3 of the algorithm require more than $O(1)$ computational operations. If $\Delta$ is the maximum degree of $T_{n_1,\ldots n_t}$, then its star chromatic index is at most $\lfloor\frac{3\Delta}{2}\rfloor$ (see Theorem~4 in \cite{class}). Therefore,    in Step~$2$, Algorithm~\ref{digraph1} runs at most $O(\Delta)$ times and the time complexity of this step  is $O(\Delta^3\log \Delta)$. Step~3 of Algorithm~\ref{2H} is transforming a solution of Problem~\ref{p2} to a solution of Problem~\ref{p1} and its time complexity is $O(\Delta)$.  Thus,   the time complexity of the algorithm is $O(\Delta^3\log \Delta)$.
 }
\end{proof}
\section{Star chromatic index of trees}\label{tree}
In this section, our goal is to find the star chromatic index of every tree and to present a polynomial time algorithm that  provides an optimum  star edge coloring of it. For this purpose,  we extend the result of Section~\ref{2tree} (for $2H$-trees) to every tree, as follows.

 Let $T$ be a  tree and  $v$ be a vertex of it.  The  induced subgraph of  $T$ on the vertices with  distance at most two from $v$ is a  $2H$-tree with root $v$, that is  denoted  by $T_v$. Clearly, 
  $$\chi^\prime_s(T)\geq \max\{\chi^\prime_s(T_v): v\in V(T)\}.$$
   In the following theorem, we show that in fact the equality holds for every tree.
\begin{theorem}\label{maxtr}
For every tree $T$, we have
\[\chi^\prime_s(T)=\max\{\chi^\prime_s(T_v): v\in V(T)\}.\] 
Moreover, there is a polynomial time algorithm to find a star edge coloring of $T$ with $\chi^\prime_s(T)$ colors.
\end{theorem}
\begin{proof}
{
Let $m=\max\{\chi^\prime_s(T_v): v\in V(T)\}$. Since $\chi^\prime_s(T)\geq m$, to show the equality, it suffices to present a star edge coloring of $T$ with $m$ colors. 
To see that, in Algorithm~\ref{T}, we  present a star edge coloring of $T$ with  $m$ colors.
The main idea of this algorithm is that for every vertex $v$ of $T$, it defines an OVGS of length $m$ that  corresponds to $T_v$ and obtains a realization for it. Then, using the  arguments in Theorem~\ref{p1p2}, the algorithm colors the edges of $T_v$ with $m$ colors. 
 
We use the following assumptions and notations in Algorithm~\ref{T}. Let $T$ be a rooted tree with root $u$. 
We denote the neighbours of every vertex $v$ in $T$, by $f_1(v), \ldots , f_{d(v)}(v)$. 
  If $v\neq u$ is a vertex in level $\ell$, then we assume that $f_1(v)$ is the parent of $v$.
Moreover, we assume that   $d(f_2(v))\leq d(f_3(v))\leq\cdots\leq  d(f_{d(v)}(v))$. If $v=u$, then we also have $d(f_1(v))\leq d(f_2(v))$. For each vertex $v$ in $T$, by $C(v)$ we mean the  set of colors of the edges incident to $v$.
\vspace*{5mm}
\hrule
\vspace*{-5mm}
\begin{algorithm1}\label{T}
An optimum star edge coloring of the given tree $T$.
\vspace*{3.5pt}
\hrule
\medskip
{\bf Step 1.} For every vertex $v$ of $T$, run Algorithm~\ref{2H} to determine $\chi^\prime_s(T_v)$.

\medskip

{\bf Step 2.} Set $m=\max\{\chi^\prime_s(T_v): v\in T\}$, and $C=\{1,\ldots,m\}$.

\medskip

{\bf Step 3.} Consider an arbitrary vertex $u$ of $T$ as the root.

\medskip

{\bf Step 4.} For $i$ from $1$ to $d(u)$ color edge $uf_i(u)$ with $i$.

\medskip

{\bf Step 5.} Set  $\ell=2$.

\medskip

{\bf Step 6.} While there exist uncolored edges in $T$ do the following steps.

\medskip

{\setlength\parindent{30pt}{\bf Step 6.1.} If there is no uncolored edge  between vertices  in level  $\ell$ and $\ell+1$, then set 
\hspace*{80pt}  $\ell=\ell+1$..

\medskip

{\bf Step 6.2.} Choose a vertex $v$ in level $\ell$ that has uncolored incident edges.

\medskip

{\bf Step 6.3.} Set $u^\prime=f_1(v)$, $t=d(u^\prime)$, $C^\prime=C\setminus C(u^\prime)$, $k=|C^\prime|$.

\medskip

{\bf Step 6.4.} For $i$ from 1 to $t$, let $n_i=|E_{f_i(u^\prime)}|$ and $q_i$ be the color of edge $u^\prime f_{i}(u^\prime)$.

\medskip
%
%

{\bf Step 6.5.} Set $D^+_k=((0,v_{p_1}),\ldots,(0,v_{p_k}),(n_1,v_{q_1}),\ldots,(n_{t},v_{q_t}))$,  where $C^\prime=\{{p_1},\ldots,{p_k}\}$.

 \medskip

{\bf Step 6.6.} If $\ell=2$, then set $i_0=1$ and $W=\{ v_{p_1},\ldots,v_{p_k}\}$. \\
\hspace*{75pt} Otherwise,   set $i_0=2$ and $W=\{ v_{p_1},\ldots,v_{p_k},v_{q_1}\}$. 
 
 \medskip
 
 {\bf Step 6.7.} While $i_0\leq t$ and  $n_{i_0}=0$, set $i_0=i_0+1$ and $W=W\cup\{v_{q_{i_0}}\}$.
 
 \medskip
 
{\bf Step 6.8.} Let $G_W$ be the graph with no edges on vertex set $\{v_1,\ldots,v_m\}$. 
 
 \medskip
 
 {\bf Step 6.9.} If $\ell>2$, then add the edges $\{\overrightarrow{v_{q_1}v_c}:c\in C(f_1(u^\prime))\setminus\{q_1\}\}$ to $G_W$.
%
%
 
%
 
 \medskip
 
{\bf Step 6.10.} For $i$ from $i_0$ to $t$ do the following steps.}
 
 \medskip
 
  {\setlength\parindent{43pt}
  
  {\bf Step 6.10.1.} $G_W$-normalize $D^+_k$.
  
  \medskip
  
  {\bf Step 6.10.2.} Set $A_i=\{v_1,\ldots,v_m\}\setminus (N^-_{G_W}(v_{q_i})\cup \{v_{q_i}\})$, and add edges from $v_{q_i}$ to $n_i$ \\
  \hspace*{103pt} vertices in~$A_i$ with the smallest possible subscripts in $D^+_k$. 
  
  \medskip
  
  
%
  {\bf Step 6.10.3.}  Color the edges of $E_{f_i(u^\prime)}$ with different subscripts of vertices in  $N^+_{G_W}(v_{q_i})$\\
\hspace*{103pt}  such that the color set of the last $k$ edges of $E_{f_i(u^\prime)}$ is  $C^\prime$.
  \medskip
  
  {\bf Step 6.10.4.}  Set $W=W\cup\{v_{q_i}\}$.}
  
  \medskip
%

\medskip

{\bf Step 7.} Return the obtained edge coloring of $T$.

\vspace*{1mm}
\hrule
\end{algorithm1}

The performance of Algorithm~\ref{T} is as follows. In Step~1, for every vertex $v$ of $T$, we apply Algorithm~\ref{2H} to  determine $\chi^\prime_s(T_v)$.  In Step~2, we define $m=\max\{\chi^\prime_s(T_v):v\in T\}$ and  color set $C=\{1,\ldots,m\}$.
In Step~3, we consider an arbitrary vertex $u$ as the root of $T$. Then in Step~4, we color the edges incident to $u$. 
In the rest of the algorithm, we consider the set of edges incident to vertices in  levels $\ell$ and $\ell+1$ ($\ell\geq2$). If there is an uncolored edge in this set, we extend the current coloring to a coloring in which the edges in this set are colored as follows.
In Step~5, we define the variable $\ell$ with initial value 2  that indicates  the smallest integer for which there is an uncolored edge incident to the vertices in level $\ell$, during the algorithm.
 While the edge coloring of $T$ is not completed,  in every iteration of Step~6, if there is no uncolored  edge incident to the vertices in level $\ell$, then   we increase the value of $\ell$ one unit in Step~6.1.
 Otherwise,  in Step~6.2, we choose a vertex $v$ in level $\ell$ with uncolored incident edges. To color the edges incident to $v$, we consider $2H$-tree $T_{u^\prime}$, where $u^\prime$ is the parent of $v$ ($u^\prime=f_1(v)$).
 In Step~6.3~to~6.5, by Theorem~\ref{p1p2}, we define the OVGS $D^+_k=((0,v_{p_1}),\ldots,(0,v_{p_k}),(n_1,v_{q_1}),\ldots,(n_{t},v_{q_t}))$ of order $m$ that corresponds to  $T_v$.  
 Note that $C^\prime=\{p_1,\ldots,p_k\}$ is the set of colors that have been not used for coloring the edges incident to $u^\prime$.

In Step~6.6 to 6.10.2, with the similar arguments in Algorithm~\ref{digraph1}, we find a realization of $D^+_k$. More precisely,
in Step~6.6, we define variable $i_0$  that indicates the smallest index of the neighbours of $u^\prime$ with positive number of uncolored incident edges. The final value of $i_0$ is determined in Step~6.7. Moreover, in Step~6.6 and  6.7,  we define the subset  $W$ of $\{v_1,\ldots,v_m\}$ that contains the vertices with indices in $C^\prime\cup\{q_i:  |C(f_i(u^\prime))|=n_i+1\}$.  
In Step~6.8, we construct graph $G_W$ on vertex set $\{v_1,\ldots,v_m\}$.
 In Step~6.9, if the chosen vertex is in level $\ell>2$, then we add edges  from $v_{q_1}$ to the vertices with subscripts in $C(f_1(u^\prime))\setminus\{q_1\}$. 
In Step~6.10.1, we first rearrange the elements of $D^+_k$, if necessary, to make it $G_W$-normal. 
In Step~6.10.2,  for  every $q_i\in C(u^\prime)$ we determine the vertices of the leftmost PON of $v_{q_i}$ and add edges from $v_{q_i}$ to them. 

In Step~6.10.3, we  color the uncolored edges incident to $f_i(u^\prime)$  by the out-neighbours of $v_{q_i}$ and then in Step~6.10.4 we add $v_{q_i}$ to $W$. Note that we color the last edges in $E_{f_i(u^\prime)}$ with colors in $C^\prime$. After completing the edge coloring of $T_{u^\prime}$, we  repeat Step~6 of the algorithm, if needed. When all edges in $T$ are colored, Algorithm~\ref{T} returns a star edge coloring of $T$ in Step~7.

We now prove that Algorithm~\ref{T} provides an optimum star edge coloring of  $T$. In this algorithm, if $\ell=2$, then  the edge coloring of $T_u$ is obtained in the same way as  in Algorithm~\ref{2H}. 
Thus, assume that $v$  is a vertex in level $l>2$, $u^\prime=f_1(v)$, $f_2(u^\prime)=v$. 
Clearly, $T_{u^\prime}$   contains all colored edges within distance at most two from $v$. Therefore, if we color the uncolored edges incident to $v$  in such a way that  no bi-colored path of length four is created in  $T_{u^\prime}$, then we guarantee that up to this step there is no bi-colored path of length four in $T$.  
We show that   the algorithm provides a star edge coloring of $T_{u^\prime}$ with $m$ colors, as follows.
 We define  OVS ${D}^+_k=((0,v_{p_1}),\ldots,(0,v_{p_k}),(n_1,v_{q_1}),\ldots,(n_t,v_{q_t}))$  (corresponding to $T_{u^\prime}$ in Step~6.5).
Note that in $T_{u^\prime}$   only the edges incident to $u^\prime$ and $f_1(u^\prime)$ have  been colored.
 Let $W_0=\{v_{p_1},\ldots,v_{p_t}\}$, $G_{W_0}$ is the graph with no edges on vertex set $\{v_1,\ldots,v_m\}$, and   for every  $1\leq i\leq t$,         $W_i=W_{i-1}\cup \{v_{q_i}\}$.
Since  for every $2\leq i\leq t$, we color the edges of $E_{f_i(u^\prime)}$ in $T_{u^\prime}$ with the leftmost PON of $v_{q_i}$ in $G_{W_{i-1}}$, by Theorem~\ref{th1}, to prove that  in Step~6.10 we achieve a realization of  $D^+_k$,  it suffices to show that  $L_{G_{W_0}}(v_{q_1})=\{\overrightarrow{v_{q_1}v_c}:c\in C(f_1(u^\prime))\setminus\{q_1\}\}$ (see Step~6.9). Obviously, $L_{G_{W_0}}(v_{q_1})=\{v_{p_1},\ldots,v_{p_k},v_{q_2},\ldots,v_{q_{t-j}}\}$, where $j=m-n_1-1$ and $j\leq t$ (if  $j>t$, then $L_{G_{W_0}}(v_{q_1})$ contains $m-j$ elements of $\{v_{p_1},\ldots,v_{p_{k}}\}$).
Note that  the color set of the edges in $E_{u^\prime}$ has been identified in the edge coloring of $T_{f_1(u^\prime)}$ and  in Step~6.10.3, we color the last $j$ edges of $E_{u^\prime}$ with  $\{q_{t-j+1},\ldots,q_t\}\subseteq C\setminus C(f_1(u^\prime))$. Hence, $C(f_1(u^\prime))$ containes the indices of the vertices in $L_{G_{W_0}}(v_{q_1})$,  as desired.

Finally, we prove that Algorithm  \ref{T}  is a polynomial time algorithm.
Let $n$ be the number of vertices of $T$. In Step~1,  Algorithm~\ref{2H} runs $n$ times to determine the value of $m$. The running time of this process is $O(n^4\log n)$. Since the other steps are   similar to the process in Algorithm~\ref{2H},  their time complexity is  at most $O(n^3\log n)$.
Therefore, the  running time of  Algorithm~\ref{T} is of order $O(n^4\log n)$.
}
\end{proof}
 
\section{Star chromatic index of certain trees}\label{upper}
In this section, we provide some tight bounds on the star chromatic index of $2H$-trees. Using these bounds we find a formula for the star chromatic index of regular $2H$-trees and the caterpillars. 

 In~\cite{class}, Bezegov{\'a} et al.  presented an algorithm that obtains a $\lfloor \frac{3\Delta}{2}\rfloor$-star edge coloring of every tree $T$ with maximum degree $\Delta$ \cite{class}. If we restrict their algorithm to the special case where $T$ is the $t$-regular $2H$-tree $T_{(t,t)}$, then we will have the following algorithm.  
\vspace*{5mm}
\hrule
\vspace*{-5mm}
\begin{algorithm1} \label{bezogva}\rm{\cite{class}}
Star edge coloring $c$ of $T_{(t,t)}$ with root $u$.
\vspace*{2pt}
\hrule
\medskip

{\bf Step 1.} For $i$ from $0$ to $t-1$, set $c(uu_{i+1}) = i+1$.

\medskip

{\bf Step 2.} For $i$ from $0$ to $t-1$ do the following steps.

\medskip

{\setlength\parindent{30pt}{\bf Step 2.1.} For $j$ from 1 to $\lfloor \frac{t}{2}\rfloor$ set $c(u_{i+1} f_{t-j+1}(u_{i+1})) = t + j$.

\medskip

{\bf Step 2.2.} For $j$ from 1 to $\lceil \frac{t}{2}\rceil-1$ do the following steps.
}
\medskip

{\setlength\parindent{36pt}{\bf Step 2.2.1.} Set $a = \left( i + j\pmod{t} \right) +1$.

\medskip

{\bf Step 2.2.2.} Color edge $u_{i+1} f_{j+1}(u_{i+1})$ with $c(u f_a(u))$.
}
\medskip

{\bf Step 3.} Return the edge coloring $c$ of $T_{(t,t)}$.
\vspace*{1mm}
\hrule
\end{algorithm1}
Algorithm~\ref{bezogva} does not always provide an optimum star edge coloring of an arbitrary tree.
In this algorithm, if we take  $M=\max\{t,n_t+1\}$, then we get $\chi^\prime_s(T_{n_1,\ldots,n_t})\leq \lfloor\frac{3M}{2}\rfloor$. In the following lemma and theorem,  we give more precise bounds for the star chromatic index of $T_{n_1,\ldots,n_t}$.

  \begin{lemma}\label{lem3}
  If $T_{(r,t)}$ is an  $r$-regular $2H$-tree, then
  \[
  \chi^\prime_s(T_{(r,t)})\leq 
  \begin{cases} 
   r+\displaystyle\left \lfloor\frac{t}{2}\right\rfloor& \text{if}\quad t\leq 2r-1,\\
   t& \text{if}\quad t\geq 2r.
    \end{cases}
\]
\end{lemma}
  \begin{proof}
{Let $u_1, \ldots, u_t$ be the neighbours of the root $u$ in $T_{(r,t)}$. To obtain the upper bounds, we present a star edge coloring for $T_{(r,t)}$ with the number of colors equals to the bound in each case. To obtain such colorings we use Algorithm \ref{bezogva}.
Let $c$ be the star edge coloring of $T_{(t,t)}$ provided by Algorithm~\ref{bezogva}.
We have two possibilities for $r$ and $t$: either $r\geq t$, or $r < t$.
In each case, we transform the coloring $c$ of $T_{(t,t)}$ to a coloring $c'$ of $T_{(r,t)}$ as follows.

If $r\geq t$, then we color the subgraph $T_{(t,t)}$  of $T_{(r,t)}$ with coloring $c$ using $t + \lfloor \frac{t}{2} \rfloor $ colors. We now have $(r-t)$ uncolored edges in each $E_{u_i}$. We use $(r-t)$ new colors for the remaining edges to extend coloring $c$ into an edge coloring $c'$ for $T_{(r,t)}$ using 
$$t +\left \lfloor \frac{t}{2} \right\rfloor  + (r-t) = r + \left\lfloor \frac{t}{2} \right\rfloor $$
colors. Note that using the new $(r-t)$ colors in $E_{u_i}$'s, $1\leq i \leq t$, does not create a bi-colored path of length four. Therefore, $c'$ is a star edge coloring of $T_{(r,t)}$. 

  \par  If $r<t$, then again we have two cases: either $t \leq 2r-1$, or $t \geq 2r$. 
  In both cases, we remove $(t-r)$ edges in each $E_{u_i}$, $1\leq i \leq t$, from  $T_{(t,t)}$ as follows. Note that in Algorithm~\ref{bezogva} the color set of the  last $\lfloor\frac{t}{2}\rfloor$ edges of each $E_{u_i}$ is $Y=\{t+1,\ldots, t+\lfloor\frac{t}{2}\rfloor\}$ (consisting of $\lfloor \frac{t}{2} \rfloor$ colors). Moreover, note that the colors of $Y$ are not used for coloring any edge $u u_i$, $1\leq i \leq t$. If $t \leq 2r-1$ or equivalently $\lfloor \frac{t}{2} \rfloor \geq (t-r)$, then assume that $A$ is fixed subset of $(t-r)$ colors from $Y$. Since the color set of every $E_{u_i}$ contains the colors in $A$, if we delete all of the edges with colors in $A$ from $E_{u_i}$, then we obtain a star coloring $c'$ of $T_{(r,t)}$ using 
$$ t + \left\lfloor \frac{t}{2} \right\rfloor  - (t-r) = r + \left\lfloor \frac{t}{2} \right\rfloor $$
colors. 
If $t\geq 2r$ or equivalently $\lfloor \frac{t}{2} \rfloor < (t-r)$, then 
we delete the $(t-r)$ edges with the largest colors from each $E_{u_i}$, for $1\leq i \leq t$.
Note that this way, all of the edges with colors in $Y$ are deleted from each $E_{u_i}$,  $1\leq i \leq t$, since colors in $Y$ are the largest colors used in $c$.
Hence, we obtain a star edge coloring $c'$ of $T_{(r,t)}$ using at most
$$t + \left\lfloor \frac{t}{2}\right \rfloor  -  \left\lfloor \frac{t}{2} \right\rfloor = t $$
colors  and the proof is complete.
  }
  \end{proof}
  
%
  \begin{theorem}\label{thbound}
  If $T_{ n_1, \ldots, n_t}$ is a $2H$-tree  and $\sigma_t = \sum_{i=1}^t n_i$, then 
  \[
  \frac{\sigma_t}{t} + \displaystyle\left\lceil\frac{t+1}{2}\right\rceil\leq \chi^\prime_s(T_{n_1,\ldots,n_t})\leq 
\begin{cases}
    n_t +1+ \displaystyle\left\lfloor\frac{t}{2}\right\rfloor & \text{if }\quad t\leq 2n_t+1, \\
    t & \text{if }\quad t\geq 2n_t+2.
\end{cases}
\]
  \end{theorem}
  \begin{proof}
  {Let $u_1, \ldots, u_t$ be the neighbours of the root $u$ in $T_{n_1,\ldots,n_t}$.
As we mentioned in Section~\ref{tree}, we know that 
 $\chi'_s(T_{n_1, \ldots, n_t}) = t + k$, for some non-negative integer $k$.
We now find a lower bound on $k$ as follows.

As we know $T_{n_1, \ldots, n_t}$ is a tree of  height at most two, and therefore it has two levels.
Let $A$ denote the set of edges in $T_{n_1, \ldots, n_t}$  incident to $u$, and $B$ denote the set of edges in $T_{n_1, \ldots, n_t}$ incident to vertices  in level $3$.
Clearly, $A\cap B = \emptyset$ and $E(T_{n_1, \ldots, n_t}) = A \cup B$; that is, $\{A, B\}$ is a partition for the edge set of $T_{n_1, \ldots, n_t}$.

Now assume that $c$ is a star edge coloring for $T_{n_1, \ldots, n_t}$ with $t+k$ colors, where the colors are taken from set $\{1, \ldots, t+k\}$. Note that set $A$ consists of exactly $t$ edges that all meet vertex $u$. Therefore, $c(A)$ must contain $t$ distinct colors.
Without loss of generality assume that $c(A) =\{1, \ldots, t\}$ and $c(uu_i) = i$, for $1\leq i\leq t$. Also, note that
every edge in $B$ receives a color either from $\{1,\ldots, t\}$, or from $\{t+1, \ldots , t+k\}$. 
Let $B_1$ denote the subset of edges in $B$ that receive a color from $\{1,\ldots, t\}$, and $B_2$ denote the subset of edges in $B$ that receive a color from $\{t+1,\ldots, t+k\}$.
Clearly, 
\begin{align}
|B_1| + |B_2|= |B|  = \sum_{i=1}^t \left( d(u_i)-1 \right)= \sum_{i=1}^t n_i = \sigma_t. \label{sigmat}
\end{align}

For $1\leq i,j\leq t$, let $S_c$ be the set of ordered pairs $(E_{u_i} , j)$ that color $j$ is used for an edge in $E_{u_i}$.
It is easy to see that there is a bijection between $S_c$ and $B_1$, and therefore, $|S_c| = |B_1|$.
Moreover, since $c$ is a star edge coloring of $T_{n_1,\ldots,n_t}$, for every    $1\leq i , j \leq t$, only one of the pairs $(E_{u_i}, j)$ and $(E_{u_j}, i)$ may belong to $S_c$.
Thus, we conclude that $|B_1|  = |S_c| \leq \frac{t(t-1)}{2}$. 
On the other hand, every color $j \in \{t+1, \ldots, t+k\}$ could have been used for coloring an edge in every $E_{u_i}$, for $1\leq i\leq t$.
Therefore, $|B_2|  \leq tk$. Hence, by \eqref{sigmat}, we have 
$$ \frac{t(t-1)}{2} + tk \geq  |B_1| + |B_2|  = |B| = \sigma_t.$$
Therefore, we conclude that
\begin{align*}
\chi'_s(T_{n_1, \ldots, n_t}) = t + k \geq \frac{\sigma_t}{t} + \frac{t+1}{2}. \label{lowerb2}
\end{align*}

Moreover,  $T_{n_1,\ldots,n_t}$ is a subgraph of $2H$-tree $T_{(n_t,t)}$. Hence by Lemma~\ref{lem3}, the upper bound for the star chromatic index of $T_{n_1,\ldots,n_t}$ is clearly established. 
  }
  \end{proof}
  
  By Lemma~\ref{lem3} and Theorem~\ref{thbound}, we have the following corollary, that shows both bounds in Theorem~\ref{thbound}  are tight.
  Note that when $t>2r$,  the maximum degree of $T_{(r,t)}$ is $t$. Hence, in this case  $\chi^\prime_s(T_{(r,t)})= t$.
  \begin{corollary}\label{regular}
  If $T_{(r,t)}$ is an  $r$-regular $2H$-tree, then 
  \[
  \chi^\prime_s(T_{(r,t)})=
  \begin{cases} 
   r+\displaystyle\left \lfloor\frac{t}{2}\right\rfloor& \text{if}\quad t\leq 2r-1,\\
   t& \text{if}\quad t\geq 2r.
    \end{cases}
\]
  \end{corollary}
  
 Our goal in the rest of this section is to find the star chromatic index of the caterpillars.
For this purpose, first we prove the following theorem.  Note that if $T_{n_1,\ldots,n_t}$ is a $2H$-tree with $t=1$, then $T_{n_1,\ldots,n_t}$ is a star and clearly, $\chi^\prime_s(T_{n_1,\ldots,n_t})=t$.
\begin{theorem}\label{ntnt}
If in a $2H$-tree $T_{n_1, \ldots, n_t}$ \rm{(}$t\geq 2$\rm{)} with root $u$ and maximum degree $\Delta$, we have $n_1 = \cdots = n_{t-2} = 0$, then  
 $$\Delta \leq \chi'_s(T_{n_1, \ldots, n_t}) \leq \Delta +1.$$
Moreover, 
$\chi'_s(T_{n_1, \ldots, n_t}) = \Delta + 1$ if and only if $d(u_{t-1}) = d(u_t) = \Delta$.
\end{theorem}

\begin{proof}
{Clearly, $\chi'_s(T_{n_1, \ldots, n_t}) \geq \Delta$ (as it holds for every graph with maximum degree $\Delta$).
For proving the upper bound note that $T_{n_{t-1}, n_t}$ is a subtree of $T_{n_1, \ldots, n_t}$.
Also, by Theorem \ref{thbound}, we have 
$$\chi'_s(T_{n_{t-1}, n_t}) \leq  n_t + 2 \leq   \Delta +1.$$
This means that there is a star edge coloring $c$ of $T_{n_{t-1}, n_t}$ with at most $\Delta +1$ colors $\{1, \ldots, \Delta +1\}$.
We now use coloring $c$ to present a star edge coloring of $T_{n_1, \ldots, n_t}$ with $\Delta + 1$ colors as follows. We first color the edges of the subtree $T_{n_{t-1}, n_t}$ of $T_{n_1, \ldots, n_t}$ with coloring $c$. It remains to color the edges $u u_i$, $1\leq i \leq t-2$.
Note that $d(u)=t \leq \Delta$.
Hence, by assigning different colors of $\{1, \ldots, \Delta + 1\} \setminus \{c(uu_{t-1}) , c(uu_t)\}$ to $u u_i$'s, $1\leq i \leq t-2$, we make sure that the edges incident to $u$ receive distinct colors.
Also, since $n_1 = \cdots = n_{t-2} = 0$, clearly we have no bi-colored path of length four.
Therefore, the obtained coloring is a star edge coloring of $T_{n_1, \ldots, n_t}$, and the upper bound is proved.

Now if $d(u_{t-1}) = d(u_t) = \Delta$, then by Corollary \ref{regular} and the above argument, we have 
$$ \Delta + 1 \geq \chi'_s(T_{n_1, \ldots, n_t}) \geq \chi'_s(T_{n_{t-1}, n_t}) = \Delta +1.$$
Thus, we conclude that in such a case, $\chi'_s(T_{n_1, \ldots, n_t}) = \Delta +1$.

For proving the converse direction, suppose that at least one of $u_{t-1}$ and $u_t$ has degree less than $\Delta$.
We claim that in this case, $\chi'_s(T_{n_1, \ldots, n_t}) = \Delta$. By monotonicity of the star chromatic index over the subgraphs of a graph, it suffices to prove the claim for the case where root $u$ is of degree $\Delta$, $n_1 = \cdots = n_{t-2} = 0$, $n_{t-1} = \Delta -2$, and $n_t = \Delta -1$. For this purpose, we define an edge coloring $c$ for $T_{n_1, \ldots, n_t}$ in which $c(uu_i) = i$, for $1\leq i \leq \Delta$ (note that here, $t = \Delta$), $\{1, \ldots, \Delta -1\}$ is the set of colors used for coloring the edges incident to $u_{t-1}$ and $\{1, \ldots, \Delta\}$ is the set of colors used for coloring edges incident to $u_t$. It is easy to check that $c$ is a $\Delta$-star edge coloring of $T_{n_1, \ldots, n_t}$. Therefore, 
$\chi'_s(T_{n_1, \ldots, n_t}) = \Delta$ in such a case.
Hence, $\chi'_s(T_{n_1, \ldots, n_t}) = \Delta +1$ if and only if both $u_{t-1}$ and $u_t$ are of degree $\Delta$.}
\end{proof}

We now use Theorem \ref{maxtr} and \ref{ntnt} to find a characterization of the star chromatic index of the caterpillars in terms of their maximum degree.
 
\begin{theorem}\label{catp}
If $T$ is a caterpillar with maximum degree $\Delta$, then
 $$\Delta \leq \chi'_s(T) \leq \Delta +1.$$
Moreover, $\chi'_s(T) = \Delta +1$ if and only if $T$ contains two vertices $u$ and $v$ of degree $\Delta$ of distance~two.
\end{theorem}

\begin{proof}
{By Theorem \ref{maxtr}, we know that $\chi'_s(T) = \max\{\chi'_s(T_v) : v\in V(T) \}$. 
It is easy to see that by definition of the caterpillars, for every $v\in V(T)$, $T_v$ is a $2H$-tree in which the root $v$ has at most two neighbours of degree at least two. 
Hence, Theorem \ref{ntnt} implies that
$$ \Delta \leq  \max\{\chi'_s(T_v) : v\in V(T) \} \leq \Delta +1.$$
Therefore, we conclude that $\Delta \leq  \chi'_s(T) \leq \Delta +1.$
Moreover, $\chi'_s(T) = \Delta +1$ if and only if $\chi'_s(T_v)  = \Delta +1$, for some $v\in  V(T)$.
By Theorem \ref{ntnt}, we can easily see that $\chi'_s(T_v)  = \Delta +1$, for some $v\in V(T)$  if and only if there are two vertices of degree $\Delta $   of distance two in $T$.
}
\end{proof}

\setlength{\baselineskip}{0.73\baselineskip}

\end{document}